\documentclass{amsart}
\usepackage{txfonts,url}
\usepackage{amssymb}
\usepackage{enumerate}

\newtheorem{thm}{Theorem}

\newtheorem{lem}{Lemma}

\newtheorem *{mainthm*}{Main Theorem}
\theoremstyle{definition}

\numberwithin{equation}{section}

\DeclareMathOperator*{\psum}{{\sum}^\ast}
\DeclareMathOperator*{\puni}{{\coprod}^\ast}
\newcommand{\ppmod}[1]{\ (\mathrm{mod}\ {#1})}
\renewcommand{\phi}{\varphi}
\renewcommand{\epsilon}{\varepsilon}

\begin{document}
\title[A remark on the sum of a prime and a square]
{A remark on the conditional estimate \\ for the sum of a prime and a square}

\author[Y. Suzuki]{Yuta Suzuki}
\address{Graduate School of Mathematics\\
Nagoya University\\
Chikusa-ku, Nagoya 464-8602, Japan.}
\email{m14021y@math.nagoya-u.ac.jp}

\date{}

\begin{abstract}
Hardy and Littlewood conjectured that every sufficiently large integer
is either a square or the sum of a prime and a square.
Let $E(x)$ be the number of positive integers up to $x\ge4$
which does not satisfy this condition.
We prove
\[E(x)\ll x^{1/2}(\log x)^A(\log\log x)^4\]
with $A=3/2$ under the Generalized Riemann Hypothesis.
This is a small improvement of the previous remarks of Mikawa (1993)
and Perelli-Zaccagnini (1995) which claims $A=4,3$ respectively.
\end{abstract}

\subjclass[2010]{Primary 11P32; Secondary 11P55.}

\keywords{Hardy-Littlewood conjecture; Circle Method.}

\maketitle

\section{Introduction}
In 1923, Hardy and Littlewood \cite[Conjecture H]{HL} conjectured that
every sufficiently large integer is either a square or the sum of a prime and a square.
For a given real positive number $x\ge4$,
let $E(x)$ be the number of positive integers up to $x$
which is neither a square nor the sum of a prime and a square.
In the present note, we consider the conditional estimate of $E(x)$
under the Generalized Riemann Hypothesis (GRH).
We always assume GRH below.

In 1985, Vinogradov \cite{Vinogradov} remarked
that his method can be used to prove
\[E(x)\ll x^{2/3+\varepsilon}\]
under GRH where $\varepsilon>0$ is an arbitrary positive constant
and the implicit constant depends only on $\varepsilon$.
However he did not publish any details of this calculations,
and from today's point of view,
his assertion is rather weak for the conditional estimates assuming GRH.
The first detailed proof was published by Mikawa \cite[Proposition]{Mikawa}
in 1993,
and he obtained
\footnote{In his paper, Mikawa only claimed that 
\[E(x)\ll x^{1/2}(\log x)^5,\]
but what he essentially proved is (\ref{Mikawa}).}
\begin{equation}
\label{Mikawa}
E(x)\ll x^{1/2}(\log x)^{4}(\log\log x)^4
\end{equation}
by the circle method.
If we assume GRH, the factor $x^{1/2}$ of this estimate
looks like the best one with our current technology.
However, as to the factor $(\log N)^4$,
we can hope some improvements beyond Mikawa's result.
Actually, in 1995, Perelli and Zaccagnini \cite[p.191]{PZ} asserted
that one can obtain
\begin{equation}
\label{PZ}
E(x)\ll x^{1/2}(\log x)^{3+\varepsilon}
\end{equation}
by refining Mikawa's calculations.
However Perelli and Zaccagnini did not give the detailed calculations
for this assertion.
\newpage
In this note, we improve these conditional estimates to the following:
\begin{thm}
\label{Main_Theorem}
Assume GRH. Then we have
\begin{equation}
\label{Main}
E(x)\ll x^{1/2}(\log x)^{3/2}(\log\log x)^4.
\end{equation}
\end{thm}

In order to prove this improvement,
we also follow Mikawa's argument as Perelli and Zaccagnini did.
There are mainly three points to refine his argument.

The first point is
the use of the technique of Languasco and Perelli \cite{LP}. 
According to them, we shall use power series as the generating functions
rather than the trigonometric polynomials which Mikawa used.
It enables us to use the explicit formula directly,
and to reduce the errors arising from
the approximation of the prime generating function.
By this method, we can obtain the result (\ref{PZ})
which Perelli and Zaccagnini asserted.
So it seems that they obtained (\ref{PZ}) in this way.

The second point is the use of the classical Jacobi's transformation formula.
Once we decide to follow the technique of Languasco and Perelli,
it is natural to use the power series
\[\widetilde{W}(\alpha)=\sum_{n=1}^\infty e^{-n^2/N}e(n^2\alpha)\]
as the generating function of squares, where $N$ is a positive integer,
$\alpha$ is a real number, and $e(\alpha):=\exp(2\pi i\alpha)$.
However this series is just a simple variant of Jaconi's $\vartheta$-function,
so we can use Jacobi's formula instead of
the Weyl estimate or the truncated Jacobi formula \cite[Theorem 4.1]{Vaughan}.
This enables us to save some more $\log$ powers.

The third point is a careful treatment on the extension of major arcs.
Mikawa estimated the errors arising from this extension
by using the ``dual'' large sieve.
However in our case, where we are asking about $\log$ powers,
his estimate is insufficient.
So we shall devide these extended arcs into two parts,
and we use the Bessel inequality besides the large sieve.

We also consider carefully
the decay of the generating function of squares on the major arcs.

\section{Notations}
Here we briefly summarize the notations which we use in this note.
Some exceptional notations are given at each occurrence.
\begin{eqnarray*}
&x\ge4&: \text{positive real numbers,}\\[-0.5mm]
&N\ge 2&: \text{positive integers,}\\[-0.5mm]
&a,q&: \text{positive integers satisfying $(a,q)=1$,}\\[-0.5mm]
&P,Q>0&: \text{positive real numbers satisfying $PQ=N$,}\\[-0.5mm]
&m,n&: \text{integers
(We impose $m,n\ge1$ when these are used as summation variables.),}\\[-0.5mm]
&p& : \text{prime numbers,}\\[-0.5mm]
&k& : \text{positive integers,}\\[-0.5mm]
&\phi(q)&: \text{the Euler totient function,}\\[-0.5mm]
&\Lambda(n)&: \text{the von Mangoldt function,}\\[-0.5mm]
&\mu(n)&: \text{the M\"obius function,}\\[-0.5mm]
&\tau(n)&: \text{the divisor function,
i.e. the number of divisors of a positive integer $n$,}\\[-0.5mm]
&\alpha&: \text{real numbers,}\\[-0.5mm]
&e(\alpha)&:=\exp(2\pi i\alpha),\\[-0.5mm]
&z&:=1/N-2\pi i\alpha\\[-0.5mm]
&(n/p)&: \text{the Legendre symbol.}
\end{eqnarray*}
\vspace{-1mm}
Moreover, the symbols
\[\psum_{a\ppmod{q}},\quad\puni_{a\ppmod{q}}\]
denotes the sum or the disjoint union over reduced residues $a$ modulo $q$
respectively.

\section{Approximation of the generating functions}
Following Languasco and Perelli,
we use the following power series for the generating functions of prime numbers
and squares:
\[\widetilde{S}(\alpha)=\sum_{n=1}^\infty\Lambda(n)e^{-n/N}e(n\alpha),\]
\[\widetilde{W}(\alpha)=\sum_{n=1}^\infty e^{-n^2/N}e(n^2\alpha).\]
In this section,
we summarize some results on the approximation for these generating functions.

For $\widetilde{S}(\alpha)$,
we use the following variant of the estimate of Languasco and Perelli.
For the proof, see \cite{LZ_many} and \cite{LP}.
Let us introduce a new argument $z=1/N-2\pi i\alpha$.
\begin{lem}
\label{LP}
Assume GRH. If $1\le q\le N$ and $|\xi|\le 1/2$, then we have
\[\psum_{a\ppmod{q}}\int_{-\xi}^\xi
\left|\widetilde{S}\left(\frac{a}{q}+\alpha\right)
-\frac{\mu(q)}{\phi(q)}\frac{1}{z}\right|^2\,d\alpha
\ll qN\xi(\log N)^2.\]
\end{lem}

For $\widetilde{W}(\alpha)$, we use Jacobi's transformation formula
in the following form:
\begin{lem}
\label{Jacobi}
For any real number $\alpha$ and any complex number $z$ with $\Re z>0$,
we have
\[\sum_{n=-\infty}^\infty \exp\left(-\pi(n+\alpha)^2z\right)
=\frac{1}{\sqrt{z}}\sum_{n=-\infty}^\infty
\exp\left(-\pi n^2/z-2\pi in\alpha\right),\]
where the branch of $\sqrt{z}$ is chosen as its value at $1$ equals $1$.
\end{lem}

This lemma is classical. For the proof, see the textbook of Montgomery and Vaughan
\cite[Theorem 10.1]{MV}.

Before applying this formula of Jacobi to $\widetilde{W}(\alpha)$,
we need to estimate the generalized Gaussian sum
\[G(a,n;q)=\sum_{k=1}^q e\left(\frac{ak^2+nk}{q}\right).\]
We abbreviate $G(a,0;q)$ by $G(a,q)$.
Our following estimate for the generalized Gaussian sum is well-known,
but we include its proof for completeness.
\begin{lem}
\label{Gauss_sum}
If $(a,q)=1$, then we have
\[G(a,n;q)\ll q^{1/2}.\]
\end{lem}

\begin{proof}
We have
\[\left|G(a,n;q)\right|^2=\sum_{k=1}^q\sum_{l=1}^q
e\left(\frac{a(k-l)(k+l)+n(k-l)}{q}\right).\]
By introducing a new argument $d=k-l$, we can rewrite the above sum as
\begin{align*}
\left|G(a,n;q)\right|^2
=&\sum_{d=1}^q\sum_{l=1}^q
e\left(\frac{ad(d+2l)+nd}{q}\right)
=\sum_{d=1}^qe\left(\frac{ad^2+nd}{q}\right)
\sum_{l=1}^qe\left(\frac{2adl}{q}\right).
\end{align*}
Here the inner sum is $q$ when $d=q/2\text{ or }q$ and
vanishes in the other cases. So the above sum is reduced to
\[\left|G(a,n;q)\right|^2\ll q.\]
Hence we obtain the lemma.
\end{proof}

We can now prove our approximation of $\widetilde{W}(\alpha)$.
The following estimate is also
classical.
For example, see \cite[Lecture 33]{Rademacher}.
However, we include its proof for completeness.
\begin{lem}
\label{theta}
If $(a,q)=1$ and $|\alpha|\le 1/2$, then we have
\[\widetilde{W}\left(\frac{a}{q}+\alpha\right)
=\frac{\sqrt{\pi}}{2}\cdot\frac{G(a,q)}{q}\cdot\frac{1}{\sqrt{z}}
+O\left(q^{1/2}+q^{1/2}N^{1/2}|\alpha|^{1/2}\right),\]
where the branch of $\sqrt{z}$ is chosen as its value at $1$ equals $1$.
\end{lem}

\begin{proof}
Let us first consider a kind of Jacobi theta function
\[\Theta(\alpha)=\sum_{n=-\infty}^\infty e^{-n^2/N}e(n^2\alpha)\]
instead of $\widetilde{W}(\alpha)$.
Then the original $\widetilde{W}(\alpha)$ can be written as
\[\widetilde{W}(\alpha)=\frac{1}{2}\Theta(\alpha)+O(1).\]
Moreover, by introducing a new argument
\[w=\frac{z}{\pi}=\frac{1}{\pi N}-2i\alpha,\]
we can rewrite this series as
\[\Theta\left(\frac{a}{q}+\alpha\right)=\sum_{n=-\infty}^\infty e^{-\pi n^2w}
e\left(\frac{an^2}{q}\right).\]

We now divide the above series according to the residues of $n$ modulo $q$.
Then we have
\begin{align*}
\Theta\left(\frac{a}{q}+\alpha\right)
=&\sum_{k=1}^qe\left(\frac{ak^2}{q}\right)
\sum_{\substack{n=-\infty\\n\equiv k\pmod{q}}}^\infty e^{-\pi n^2w}
=\sum_{k=1}^qe\left(\frac{ak^2}{q}\right)
\sum_{m=-\infty}^\infty \exp\left(-\pi (m+k/q)^2q^2w\right).
\end{align*}
Remembering the condition $\Re w>0$,
we apply Lemma \ref{Jacobi} to the inner sum.
As a result, we get
\begin{align*}
\Theta\left(\frac{a}{q}+\alpha\right)
=&\sum_{k=1}^qe\left(\frac{ak^2}{q}\right)
\frac{1}{q\sqrt{w}}\sum_{n=-\infty}^\infty
\exp\left(-\frac{\pi n^2}{q^2w}-2\pi i\frac{nk}{q}\right)\\
=&\frac{1}{q\sqrt{w}}\sum_{n=-\infty}^\infty
\exp\left(-\frac{\pi n^2}{q^2w}\right)
\sum_{k=1}^q e\left(\frac{ak^2+nk}{q}\right).
\end{align*}
Recalling the definition of the generalized Gaussian sum,
we now pick up the term with $n=0$ as the main term.
This results that
\begin{align}
\Theta\left(\frac{a}{q}+\alpha\right)=&\frac{1}{q\sqrt{w}}\sum_{n=-\infty}^\infty
G(a,n;q)\exp\left(-\frac{\pi n^2}{q^2w}\right)\notag\\
\label{divided}
=&\frac{1}{q\sqrt{w}}\left\{G(a,q)+R(a,q;\alpha)\right\}
\end{align}
where 
\[R(a,q;\alpha)=\sum_{\substack{n=-\infty\\n\neq0}}^\infty
G(a,n;q)\exp\left(-\frac{\pi n^2}{q^2w}\right).\]
The remaining task is to estimate this error term $R(a,q;\alpha)$.
By Lemma \ref{Gauss_sum}, we can estimate as
\begin{equation}
R(a,q;\alpha)=\sum_{\substack{n=-\infty\\n\neq0}}^\infty
G(a,n;q)\exp\left(-\frac{\pi n^2}{q^2w}\right)
\ll\sqrt{q}\sum_{n=1}^\infty
\exp\left(-\frac{\pi n^2\delta}{q^2}\right),
\end{equation}
where the argument $\delta$ is given by
\[\delta=\Re\left(\frac{1}{w}\right)
=\frac{1}{\pi N}\cdot\frac{1}{1/(\pi N)^2+4\alpha^2}\asymp\frac{1}{N|w|^2}.\]
Then
\begin{align*}
R(a,q;\alpha)\ll&\sqrt{q}\sum_{n=1}^\infty
\exp\left(-\frac{\pi n^2\delta}{q^2}\right)
\ll\sqrt{q}\int_0^\infty\exp\left(-\frac{\pi t^2\delta}{q^2}\right)dt
\ll q^{3/2}N^{1/2}|w|.
\end{align*}
Substituting this estimate into (\ref{divided}), we get
\begin{align*}
\Theta\left(\frac{a}{q}+\alpha\right)
=\frac{G(a,q)}{q\sqrt{w}}
+O\left(q^{1/2}N^{1/2}|w|^{1/2}\right).
\end{align*}
Returning to the original notations $\widetilde{W}(\alpha)$ and $z=\pi w$,
we can rewrite this as
\[\widetilde{W}\left(\frac{a}{q}+\alpha\right)=
\frac{\sqrt{\pi}}{2}\cdot\frac{G(a,q)}{q}\cdot\frac{1}{\sqrt{z}}
+O\left(1+q^{1/2}N^{1/2}|z|^{1/2}\right).\]
Finally, since $|z|^{1/2}\ll N^{-1/2}+|\alpha|^{1/2}$,
we obtain the lemma.
\end{proof}

\section{Proof of Theorem \ref{Main_Theorem}}
Let us consider
\[R(n)=\sum_{k+m^2=n}\Lambda(k),\]
which is a variant of the representation function for the sum of a prime and a square.
Then this representation function $R(n)$ appears as the coefficient
in the following power series:
\[\widetilde{S}(\alpha)\widetilde{W}(\alpha)
=\sum_{n=1}^\infty R(n)e^{-n/N}e(n\alpha).\]
We also define the singular series for our problem
\[\mathfrak{S}(n)=
\begin{cases}
\displaystyle \prod_{p>2}\left(1-\frac{(n/p)}{p-1}\right) &\text{(when $n$ is not a square)},\\
\displaystyle 0 &\text{(when $n$ is a square)},
\end{cases}\]
which was first introduced by Hardy and Littlewood \cite[Conjecture H]{HL}.

The main body of our argument is to estimate the following mean square error
as usual in the estimates of the exceptional sets.
\begin{thm}
\label{mean_thm}
Assume GRH. If $N\ge 2$, then we have
\begin{equation}
\label{yuta_L2}
\sum_{n\le N}
\left|R(n)-\mathfrak{S}(n)\sqrt{n}(1+O(n^{-\eta}))\,\right|^2
\ll (N\log N)^{3/2}
\end{equation}
for some positive absolute constant $\eta>0$.
\end{thm}

\begin{proof}
We use the circle method. We shall start with giving the Farey dissection.

Let us take $Q=(N\log N)^{1/2}$ as the order of the dissection,
and let $I$ be the modified unit interval, i.e.
\[I=(1/Q,1+1/Q].\]
For a pair of positive integers $a,q$ such that $a\le q$ and $(a,q)=1$,
we shall denote by $\xi_{a,q}$ the Farey arcs which is defined by
\[\xi_{a,q}=\left(\frac{a}{q}-\frac{1}{qQ}, \frac{a}{q}+\frac{1}{qQ}\right].\]
Let us introduce $P=N/Q$.
Then the Farey arcs $\xi_{a,q}$
with $q\le P$ are pairwise disjoint for sufficiently large $N$.
So let us define the major arcs:
\[\mathfrak{M}=\coprod_{q\le P}\puni_{a\ppmod{q}}\xi_{a,q}\]
and define the minor arcs:
\[\mathfrak{m}=I\setminus\mathfrak{M}.\]

Let us introduce the following approximants for our generating functions
\begin{equation}
T(\alpha):=
\begin{cases}
\displaystyle
\vspace{3mm}
\frac{\mu(q)}{\phi(q)}\cdot\frac{1}{(z+2\pi ia/q)}
&\left(\text{when $\alpha\in\xi_{a,q},\ q\le P$}\right),\\
0&\left(\text{when $\alpha\in\mathfrak{m}$}\right),
\end{cases}
\end{equation}
\begin{equation}
U(\alpha):=
\begin{cases}
\displaystyle
\vspace{3mm}
\frac{\sqrt{\pi}}{2}
\cdot\frac{G(a,q)}{q}\cdot\frac{1}{(z+2\pi ia/q)^{1/2}}
&\left(\text{when $\alpha\in\xi_{a,q},\ q\le P$}\right),\\
0&\left(\text{when $\alpha\in\mathfrak{m}$}\right).
\end{cases}
\end{equation}
Consider the Fourier coefficients $\widehat{TU}(n)$ of $T(\alpha)U(\alpha)$
which is defined by
\begin{equation}
\label{fourier}
e^{-n/N}\widehat{TU}(n)=
\int_0^1T(\alpha)U(\alpha)e(-n\alpha)\,d\alpha.
\end{equation}
Then we get by the Parseval identity that
\begin{equation}
\label{parseval}
\sum_{n=1}^\infty e^{-2n/N}\left|R(n)-\widehat{TU}(n)\right|^2
=\int_0^1
\left|\widetilde{S}(\alpha)\widetilde{W}(\alpha)
-T(\alpha)U(\alpha)\right|^2\,d\alpha.
\end{equation}
We devide this last integral as
\begin{equation}
\label{arc_decomp}
\int_0^1
\left|\widetilde{S}(\alpha)\widetilde{W}(\alpha)
-T(\alpha)U(\alpha)\right|^2\,d\alpha
=\int_\mathfrak{M}+\int_\mathfrak{m}
\end{equation}
into the two integrals over the major arcs and the minor arcs,
and then we shall estimate these integrals separately.

We first treat the integral over the minor arcs.
On the minor arcs, the argument $\alpha$ stands in the range
\[\left|\alpha-\frac{a}{q}\right|\le\frac{1}{qQ}\]
with some Farey fraction $a/q$ with $P<q\le Q$ by Dirichlet's theorem.
So by Lemmas \ref{Gauss_sum} and \ref{theta},
we get (see (\ref{z_estimate}) below)
\[\widetilde{W}(\alpha)^2\ll \frac{N}{P}+Q+\frac{N}{Q}\ll Q.\]
Hence by the Parseval identity, we have
\begin{align*}
\int_\mathfrak{m}
=&\int_\mathfrak{m}
\left|\widetilde{S}(\alpha)\widetilde{W}(\alpha)\right|^2\,d\alpha
\ll\sup_{\alpha\in\mathfrak{m}}\left|\widetilde{W}(\alpha)\right|^2
\int_0^1
\left|\widetilde{S}(\alpha)\right|^2\,d\alpha
\ll Q\sum_{n=1}^\infty\Lambda(n)^2e^{-2n/N}.
\end{align*}
Using the prime number theorem, we obtain the following estimate for the minor arcs:
\begin{equation}
\label{minor_arc}
\int_\mathfrak{m}\ll QN(\log N)=(N\log N)^{3/2}.
\end{equation}

We next estimate the integral over the major arcs.
We divide the integrand as
\begin{align*}
\widetilde{S}(\alpha)\widetilde{W}(\alpha)-T(\alpha)U(\alpha)
=&\left(\widetilde{S}(\alpha)\widetilde{W}(\alpha)-
T(\alpha)\widetilde{W}(\alpha)\right)
+\left(T(\alpha)\widetilde{W}(\alpha)-T(\alpha)U(\alpha)\right)\\
=&E_1+E_2,\text{ say.}
\end{align*}
And we estimate the integral over major arcs by separating these two errors as
\[\int_\mathfrak{M}
\ll\int_\mathfrak{M}\left|E_1\right|^2d\alpha
+\int_\mathfrak{M}\left|E_2\right|^2d\alpha=I_1+I_2.\]
Our estimates for both of these integrals start with
the dissection of them into small arcs:
\[\int_\mathfrak{M}=\sum_{q\le P}\psum_{a\ppmod{q}}\int_{\xi_{a,q}}.\]

We first treat $I_1$.
By Lemmas \ref{Gauss_sum} and \ref{theta},
for any Farey fraction $a/q$ with $q\le P$, we have
\[\widetilde{W}\left(\frac{a}{q}+\alpha\right)^2
\ll\frac{1}{\sqrt{q|z|}}+q^{1/2}+q^{1/2}N^{1/2}|\alpha|^{1/2}.\]
We can estimate $|z|$ from below by its real and imaginary parts:
\begin{equation}
\label{z_estimate}
|z|\gg\max\left(\frac{1}{N},|\alpha|\right).
\end{equation}
Hence we obtain the following formulae:
\begin{gather}
\label{center_arc}
\widetilde{W}\left(\frac{a}{q}+\alpha\right)^2\ll\frac{N}{q},\\
\label{side_arc}
\widetilde{W}\left(\frac{a}{q}+\alpha\right)^2
\ll\frac{1}{q|\alpha|}+q+qN|\alpha|\ll\frac{1}{q|\alpha|},
\end{gather}
for any Farey fraction $a/q$ with $q\le P$ and $|\alpha|\le1/qQ$.
The integral which we have to estimate is
\[I_1=\sum_{q\le P}\psum_{a\ppmod{q}}
\int_{-1/qQ}^{1/qQ}\left|\widetilde{S}\left(\frac{a}{q}+\alpha\right)
-\frac{\mu(q)}{\phi(q)}\frac{1}{z}\right|^2
\left|\widetilde{W}\left(\frac{a}{q}+\alpha\right)\right|^2d\alpha.\]
We shall estimate each of these integrals by dividing into two parts as
\begin{align*}
\int_{-1/qQ}^{1/qQ}
=\int_{-1/N}^{1/N}+\int_{1/N<|\alpha|\le1/qQ}
=J_1(a,q)+J_2(a,q).
\end{align*}
For $J_1(a,q)$, we use the estimate (\ref{center_arc}) and Lemma \ref{LP}.
This results
\begin{align}
\sum_{q\le P}\psum_{a\ppmod{q}}J_1(a,q)\notag
=&\sum_{q\le P}\psum_{a\ppmod{q}}
\int_{-1/N}^{1/N}\left|\widetilde{S}\left(\frac{a}{q}+\alpha\right)
-\frac{\mu(q)}{\phi(q)}\frac{1}{z}\right|^2
\left|\widetilde{W}\left(\frac{a}{q}+\alpha\right)\right|^2d\alpha\notag\\
\ll&\sum_{q\le P}\frac{N}{q}\psum_{a\ppmod{q}}
\int_{-1/N}^{1/N}\left|\widetilde{S}\left(\frac{a}{q}+\alpha\right)
-\frac{\mu(q)}{\phi(q)}\frac{1}{z}\right|^2d\alpha\notag\\
\label{major_arc_11}
\ll&PN(\log N)^2.
\end{align}
For $J_2(a,q)$, we use the estimate (\ref{side_arc}) and Lemma \ref{LP}.
In order to apply (\ref{side_arc}), we employ the following dyadic subdivision:
\begin{align*}
\int_{1/N<|\alpha|\le1/qQ}\ll
\sum_{k=1}^K\int_{\eta_k<|\alpha|\le2\eta_k}
+\int_{1/N<|\alpha|\le\eta_K},
\end{align*}
where $\eta_k$'s are given by
\[\eta_k:=\frac{1}{2^kqQ}\ (k=1,2,3,\dots),\quad \eta_{K+1}<1/N\le\eta_K.\]
Then for $\eta\in[1/2N,1/qQ]$,
(\ref{side_arc}) and Lemma \ref{LP} gives the estimate
\begin{align*}
&\psum_{a\ppmod{q}}\int_{\eta<|\alpha|\le 2\eta}
\left|\widetilde{S}\left(\frac{a}{q}+\alpha\right)
-\frac{\mu(q)}{\phi(q)}\frac{1}{z}\right|^2
\left|\widetilde{W}\left(\frac{a}{q}+\alpha\right)\right|^2d\alpha\\
\ll&\psum_{a\ppmod{q}}\int_{\eta<|\alpha|\le 2\eta}\frac{1}{q|\alpha|}
\left|\widetilde{S}\left(\frac{a}{q}+\alpha\right)
-\frac{\mu(q)}{\phi(q)}\frac{1}{z}\right|^2d\alpha\\
\ll&\frac{qN\eta(\log N)^2}{q\eta}
=N(\log N)^2.
\end{align*}
Since our $K$ satisfies
\[1/N\le\eta_K=\frac{1}{2^KqQ}\]
or
\[K+1\ll\log\frac{2N}{qQ},\]
we have
\begin{equation}
\label{major_arc_12}
\sum_{q\le P}\psum_{a\ppmod{q}}J_2(a,q)
\ll N(\log N)^2\sum_{q\le P}\log\frac{2N}{qQ}
\ll PN(\log N)^2.
\end{equation}
Here we used the following estimate:
\[\sum_{q\le P}\log\frac{2N}{qQ}
=\sum_{q\le P}\log\frac{2P}{q}
=\sum_{q\le P}\int_{q}^{2P}\frac{du}{u}\ll P.\]
By (\ref{major_arc_11}) and (\ref{major_arc_12}), we have
\begin{equation}
\label{major_arc_1}
I_1\ll PN(\log N)^2=(N\log N)^{3/2}.
\end{equation}

We now estimate $I_2$.
On each small arc $\xi_{a,q}$, Lemma \ref{theta} gives that
\[\left|\widetilde{W}(\alpha)-U(\alpha)\right|^2\ll P.\]
And since we have
\begin{align*}
\int_{\xi_{a,q}}\left|T(\alpha)\right|^2d\alpha
\ll&\frac{1}{\phi(q)^2}\int_{-1/qQ}^{1/qQ}\frac{d\alpha}{|z|^2}\\
\ll&\frac{1}{\phi(q)^2}
\left(\int_{|\alpha|\le\min\left(1/N,1/qQ\right)}\frac{d\alpha}{|z|^2}
+\int_{1/N<|\alpha|\le1/qQ}\frac{d\alpha}{|z|^2}\right)\\
\ll&\frac{1}{\phi(q)^2}\left(N
+\int_{1/N}^{1/qQ}\frac{d\alpha}{\alpha^2}\right)
\ll\frac{N}{\phi(q)^2},
\end{align*}
we can estimate $I_2$ as
\begin{align}
I_2
\ll&P\sum_{q\le P}\psum_{a\ppmod{q}}
\int_{\xi_{a,q}}\left|T(\alpha)\right|^2\,d\alpha\notag\\
\label{major_arc_2}
\ll&PN\sum_{q\le P}\frac{1}{\phi(q)}\ll PN(\log N)=N^{3/2}(\log N)^{1/2}.
\end{align}

The above estimates (\ref{major_arc_1}), (\ref{major_arc_2}) of $I_1, I_2$
give
\[\int_\mathfrak{M}\ll (N\log N)^{3/2}\]
for the integral over the major arcs.
Now, we combine (\ref{minor_arc}) with this inequality
and recall (\ref{parseval}) and (\ref{arc_decomp}) to get
\[\sum_{n=1}^\infty e^{-2n/N}\left|R(n)-\widehat{TU}(n)\right|^2
\ll (N\log N)^{3/2}.\]
In particular, we have
\begin{equation}
\label{pre_mean}
\sum_{n\le N}\left|R(n)-\widehat{TU}(n)\right|^2
\ll (N\log N)^{3/2}
\end{equation}
since $e^{-2n/N}\gg1$ for $n\le N$.

The next task is to extend each small arcs to the whole arc $I$
in the Fourier coefficients formula (\ref{fourier}).
Recalling that $T(\alpha)$ and $U(\alpha)$ have non-zero values
only on the major arcs,
we can divide (\ref{fourier}) into the integrals over small arcs:
\begin{align*}
e^{-m/N}\widehat{TU}(m)
=&\frac{\sqrt{\pi}}{2}\sum_{q\le P}\frac{\mu(q)}{q\phi(q)}\psum_{a\ppmod{q}}G(a,q)
\int_{\xi_{a,q}}\frac{e(-m\alpha)}{(z+2\pi ia/q)^{3/2}}\,d\alpha\\
=&\frac{\sqrt{\pi}}{2}\sum_{q\le P}\frac{\mu(q)}{q\phi(q)}
\psum_{a\ppmod{q}}G(a,q)e\left(-\frac{am}{q}\right)
\int_{-1/qQ}^{1/qQ}\frac{e(-m\alpha)}{z^{3/2}}\,d\alpha
\end{align*}
where $z^{3/2}$ is defined by
\[z^{3/2}=\exp\left(\frac{3}{2}\log z\right)\]
with the principal branch of $\log z$.
Here we have to extend the range of the integral
\[\int_{-1/qQ}^{1/qQ}\frac{e(-m\alpha)}{z^{3/2}}\,d\alpha\]
to the whole arc $[-1/2,1/2]$. The errors arising from this extension is
\[r_m=\frac{\sqrt{\pi}}{2}\sum_{q\le P}\frac{\mu(q)}{q\phi(q)}
\psum_{a\ppmod{q}}G(a,q)e\left(-\frac{am}{q}\right)
\int_{1/qQ<|\alpha|\le1/2}
\frac{e(-m\alpha)}{z^{3/2}}\,d\alpha,\]
and we shall estimate its squared mean value
\[\sum_{m\le N}|r_m|^2.\]
We divide each of the above integrals into two parts as
\[\int_{1/qQ<|\alpha|\le1/2}
=\int_{1/qQ<|\alpha|\le1/4qP}+\int_{1/4qP<|\alpha|\le1/2}.\]
Then the former small extended arcs
\[\left[\frac{a}{q}-\frac{1}{4qP},\frac{a}{q}-\frac{1}{qQ}\right],\quad
\left[\frac{a}{q}+\frac{1}{qQ},\frac{a}{q}+\frac{1}{4qP}\right]\]
for $q\le P$ are pairwise disjoint
since for two distinct Farey fraction $a/q, a'/q'$ with $q,q'\le P$, we have
\[\left|\frac{a}{q}-\frac{a'}{q'}\right|\ge\frac{1}{qq'}
\ge\frac{1}{2qP}+\frac{1}{2q'P}.\]
So let us introduce
\[\mathfrak{n}=\coprod_{q\le P}\puni_{a\ppmod{q}}\left(
\left[\frac{a}{q}-\frac{1}{4qP},\frac{a}{q}-\frac{1}{qQ}\right]
\cup
\left[\frac{a}{q}+\frac{1}{qQ},\frac{a}{q}+\frac{1}{4qP}\right]\right).\]
Then we find that
\[r_m=r_m^{(1)}+r_m^{(2)},\]
where
\[r_m^{(1)}=\int_{\mathfrak{n}}T(\alpha)U(\alpha)e(-m\alpha)d\alpha,\]
\[r_m^{(2)}=\frac{\sqrt{\pi}}{2}\sum_{q\le P}\frac{\mu(q)}{q\phi(q)}
\psum_{a\ppmod{q}}G(a,q)e\left(-\frac{am}{q}\right)
\int_{1/4qP<|\alpha|\le1/2}
\frac{e(-m\alpha)}{z^{3/2}}\,d\alpha.\]

We first treat $r_m^{(1)}$.
For this integral, we can use the Bessel inequality, and we get
\[\sum_{m\le N}\left|r_m^{(1)}\right|^2
\ll\int_{\mathfrak{n}}\left|T(\alpha)U(\alpha)\right|^2d\alpha.\]
Again returning to small arcs, we have
\begin{align*}
\sum_{m\le N}\left|r_m^{(1)}\right|^2
\ll&\sum_{q\le P}\frac{\mu^2(q)}{q^2\phi(q)^2}
\psum_{a\ppmod{q}}\left|G(a,q)\right|^2
\int_{1/qQ<|\alpha|\le1/4qP}
\frac{1}{|z|^3}d\alpha\\
=&\sum_{q\le P}\frac{\mu^2(q)}{q\phi(q)}
\int_{1/qQ<|\alpha|\le1/4qP}\frac{1}{|\alpha|^3}d\alpha
\ll Q^2\sum_{q\le P}\frac{q}{\phi(q)}\ll Q^2P=QN.
\end{align*}
Thus we have
\begin{equation}
\label{extension_neighbouring}
\sum_{m\le N}\left|r_m^{(1)}\right|^2\ll N^{3/2}(\log N)^{1/2}.
\end{equation}

We now deal with $r_m^{(2)}$.
Following Mikawa \cite{Mikawa}, we use the large sieve for these integrals.
Our mean square error
\[\sum_{m\le N}\left|r_m^{(2)}\right|^2\]
is a constant multiple of
\[\sum_{m\le N}\left|\sum_{q\le P}\frac{\mu(q)}{q\phi(q)}
\psum_{a\ppmod{q}}G(a,q)e\left(-\frac{am}{q}\right)
\int_{1/4qP<|\alpha|\le 1/2}\frac{e(-m\alpha)}{z^{3/2}}\,d\alpha\right|^2.\]
In the absolute sign of each above term,
we first change the order of summation and integration:
\begin{align*}
&\sum_{q\le P}\frac{\mu(q)}{q\phi(q)}
\psum_{a\ppmod{q}}G(a,q)e\left(-\frac{am}{q}\right)
\int_{1/4qP<|\alpha|\le 1/2}\frac{e(-m\alpha)}{z^{3/2}}\,d\alpha\\
=&
\int_{1/N<|\alpha|\le1/2}\frac{e(-m\alpha)}{z^{3/2}}
\sum_{1/4|\alpha|P<q\le P}\frac{\mu(q)}{q\phi(q)}
\psum_{a\ppmod{q}}G(a,q)e\left(-\frac{am}{q}\right)\,d\alpha.
\end{align*}
Then by the Cauchy-Schwarz inequality, we get
\begin{align*}
&\left|\sum_{q\le P}\frac{\mu(q)}{q\phi(q)}
\psum_{a\ppmod{q}}G(a,q)e\left(-\frac{am}{q}\right)
\int_{1/4qP<|\alpha|\le 1/2}\frac{e(-m\alpha)}{z^{3/2}}\,d\alpha\right|^2\\
\ll&\left(\int_{1/N<|\alpha|\le1/2}\frac{1}{|z|}d\alpha\right)
\left(\int_{1/N<|\alpha|\le 1/2}\frac{1}{|z|^2}
\left|\sum_{1/4|\alpha|P<q\le P}\frac{\mu(q)}{q\phi(q)}
\psum_{a\ppmod{q}}G(a,q)e\left(-\frac{am}{q}\right)\right|^2\,d\alpha\right).
\end{align*}
The first integral can be estimated as
\[\int_{1/N<|\alpha|\le1/2}\frac{1}{|z|}d\alpha
\ll\int_{1/N<|\alpha|\le1/2}\frac{1}{|\alpha|}d\alpha
\ll\log N.\]
And for the second integral, after taking the summation over $m$,
we use the following estimate obtained through the ``dual'' large sieve:
\begin{align*}
&\sum_{m\le N}
\left|\sum_{1/4|\alpha|P<q\le P}\frac{\mu(q)}{q\phi(q)}
\psum_{a\ppmod{q}}G(a,q)e\left(-\frac{am}{q}\right)\right|^2\\
\ll&\left(N+P^2\right)\sum_{1/4|\alpha|P<q\le P}
\psum_{a\ppmod{q}}\frac{\mu^2(q)\left|G(a,q)\right|^2}{q^2\phi(q)^2}
\ll N\sum_{1/|\alpha|P<q\le P}\frac{\mu^2(q)}{q\phi(q)}.
\end{align*}
Combining the above estimates, we get the estimate
\begin{align*}
&\sum_{m\le N}\left|\sum_{q\le P}\frac{\mu(q)}{q\phi(q)}
\psum_{a\ppmod{q}}G(a,q)e\left(-\frac{am}{q}\right)
\int_{1/4qP<|\alpha|\le 1/2}
\frac{e(-m\alpha)}{z^{3/2}}\,d\alpha\right|^2\\
\ll&N(\log N)\int_{1/N<|\alpha|\le 1/2}\frac{1}{|\alpha|^2}
\sum_{1/4|\alpha|P<q\le P}\frac{\mu^2(q)}{q\phi(q)}d\alpha\\
\ll&N(\log N)\sum_{q\le P}\frac{\mu^2(q)}{q\phi(q)}
\int_{1/4qP<|\alpha|\le 1/2}\frac{d\alpha}{|\alpha|^2}\\
\ll&PN(\log N)\sum_{q\le P}\frac{\mu^2(q)}{\phi(q)}\ll PN(\log N)^2
\end{align*}
or
\begin{equation}
\label{extension_remained}
\sum_{m\le N}\left|r_m^{(2)}\right|^2\ll (N\log N)^{3/2}.
\end{equation}

Combining (\ref{extension_neighbouring}) and (\ref{extension_remained}), we have
\begin{equation}
\label{expand_error}
\sum_{m\le N}\left|r_m\right|^2
\ll\sum_{m\le N}\left|r_m^{(1)}\right|^2+\sum_{m\le N}\left|r_m^{(2)}\right|^2
\ll(N\log N)^{3/2}.
\end{equation}

We next calculate explicitly the extended integral
\[V(m,P)=\left\{\frac{\sqrt{\pi}}{2}\sum_{q\le P}\frac{\mu(q)}{q\phi(q)}
\psum_{a\ppmod{q}}G(a,q)e\left(-\frac{am}{q}\right)\right\}
\int_{-1/2}^{1/2}\frac{e(-m\alpha)}{z^{3/2}}d\alpha.\]
We use the following integral formula
\begin{equation}
\label{integral_formula}
\int_{-1/2}^{1/2}\frac{e(-n\alpha)}{z^{3/2}}d\alpha
=\displaystyle e^{-n/N}\frac{2\sqrt{n}}{\Gamma(1/2)}+O\left(\frac{1}{n}\right).
\end{equation}
This is a simple corollary of the Hankel integral formula.
For the proof, see \cite[Lemma 4]{LZ_recent}.
Since it holds that \cite[Lemma 3, 4]{BPP}
\[\psum_{a\ppmod{q}}G(a,q)e\left(-\frac{am}{q}\right)\ll q\]
for any square-free number $q$, the error term arising from
the error term of (\ref{integral_formula}) is less than
\[\ll\frac{1}{m}\sum_{q\le P}\frac{1}{\phi(q)}\ll\frac{\log N}{m}.\]
So $V(m,P)$ is calculated explicitly as
\begin{align}
V(m,P)
\label{coefficient_TU}
=&e^{-m/N}\sqrt{m}\mathfrak{S}(m,P)+O\left(\frac{\log N}{m}\right)
\end{align}
where
\[\mathfrak{S}(n,P):=\sum_{q\le P}
\psum_{a\ppmod{q}}\frac{\mu(q)}{q\phi(q)}G(a,q)e\left(\frac{-an}{q}\right).\]
By (\ref{pre_mean}) and (\ref{expand_error}), we get
\[\sum_{n\le N}\left|R(n)-\sqrt{n}\mathfrak{S}(n,P)\right|^2
\ll (N\log N)^{3/2}.\]
It can be shown \cite[Lemma 21]{BPP}, under the assumption of GRH
that there exists a positive absolute constant $\eta$
\[\mathfrak{S}(n,P)=\mathfrak{S}(n)+O(n^{-\eta}).\]
Therefore Theorem \ref{mean_thm} follows.
\end{proof}

\begin{proof}[Proof of Theorem \ref{Main_Theorem}]
First we shall modify Theorem \ref{mean_thm} to
\begin{equation}
\label{modified_mean}
\sum_{n\le N}
\left|\sum_{p+m^2=n}(\log p)-\mathfrak{S}(n)\sqrt{n}(1+O(n^{-\eta}))\right|^2
\ll (N\log N)^{3/2}.
\end{equation}
In order to justify this modification, we proceed as follows.
We first remove the restriction that $p$ is a prime
from the error which we have to estimate:
\[\sum_{n\le N}\left(\sum_{\substack{p^k+m^2=n\\k\ge 2}}(\log p)\right)^2
\ll(\log N)^2\left(\sum_{n\le N}\left(\sum_{m_1^2+m_2^2=n}1\right)^2
+\sum_{n\le N}
\left(\sum_{\substack{m_1^k+m_2^2=n\\k\ge 3}}1\right)^2\right).\]
Let us call these sum as
\[=(\log N)^2\left(\Sigma_1+\Sigma_2\right).\]
For $\Sigma_1$, we use Jacobi's two-square theorem to get
\[\sum_{m_1^2+m_2^2=n}1
=\sum_{\substack{d|n\\d:\,\text{odd}}}(-1)^{(d-1)/2}+O(1)\ll\tau(n).\]
This estimate gives the following estimate for $\Sigma_1$:
\[\Sigma_1\ll\sum_{n\le N}\tau(n)^2\ll N(\log N)^3.\]
For $\Sigma_2$, we first notice that for $k\ge3$ and $n\le N$
\[\sum_{\substack{m_1^k+m_2^2=n}}1
\ll\sum_{m^k\le n}1\ll\sum_{m^3\le n}1\ll N^{1/3}.\]
Hence we get
\[\Sigma_2\ll
N^{1/3}\sum_{n\le N}\sum_{\substack{m_1^k+m_2^2=n\\k\ge 3}}1.\]
The largest value of $k$ which we have to consider is at most $O(\log N)$.
So we have
\begin{align*}
\Sigma_2&\ll
N^{1/3}\sum_{k=3}^{O(\log N)}
\sum_{n\le N}\sum_{\substack{m_1^k+m_2^2=n}}1
=N^{1/3}\sum_{k=3}^{O(\log N)}\sum_{m_1^k+m_2^2\le N}1\\
&\ll N^{1/3}\sum_{k=3}^{O(\log N)}\sum_{m_1^3+m_2^2\le N}1
\ll N^{7/6}(\log N).
\end{align*}
These estimates gives
\[\sum_{n\le N}\left(\sum_{\substack{p^k+m^2=n\\k\ge 2}}(\log p)\right)^2
\ll N^{7/6}(\log N)\ll(N\log N)^{3/2}.\]
This justifies (\ref{modified_mean}).

Now assuming GRH, we know that
\[\mathfrak{S}(n)\gg(\log\log n)^{-2}\]
holds for $n\neq m^2$ \cite[p. 304]{Mikawa}.
Hence (\ref{modified_mean}) gives
\[N(\log\log N)^{-4}\sum_{\substack{N/2<n\le N\\n\neq m^2,p+m^2}}1
\ll\sum_{n\le N}
\left|\sum_{p+m^2=n}(\log p)-\mathfrak{S}(n)\sqrt{n}(1+O(n^{-\eta}))\right|^2
\ll (N\log N)^{3/2}\]
or
\[E(N)-E(N/2)\ll N^{1/2}(\log N)^{3/2}(\log\log N)^4.\]
Therefore we finally obtain that
\[E(N)\ll N^{1/2}+\sum_{k=1}^{O(\log N)}\left(E(N/2^{k-1})-E(N/2^k)\right)
\ll N^{1/2}(\log N)^{3/2}(\log\log N)^4.\]
The restriction of the variable that $N$ is a positive integer,
can be removed trivially.
\end{proof}
\subsection*{Acknowledgements}
The author would like to thank Professor Kohji Matsumoto
for his comments and successive encouragements.

\end{document}